\documentclass[12pt]{article}

\usepackage{amsmath}
\usepackage{amssymb}

\usepackage[T1]{fontenc} 
\usepackage{sets pace}
\usepackage{comment}
\usepackage{colortbl}
\definecolor{text1}{cmyk}{1,.65,0,0} 
\definecolor{text2}{rgb}{1,0,0} 
\definecolor{text3}{cmyk}{0,0,0,1} 
\definecolor{text4}{cmyk}{0,0,0,0.5} 
\definecolor{text5}{cmyk}{1.0,0.0,1.0,0} 

\usepackage{amsfonts}
\usepackage{amsthm}
\usepackage{graphicx}
\usepackage{caption}
\usepackage{subcaption}
\usepackage{latexsym}
\usepackage[most]{tcolorbox}
\makeatletter
\renewcommand{\@seccntformat}[1]
{\csname the#1\endcsname.\enspace}
\makeatother
\newtheorem{definition}{Definition}
\newtheorem{theorem}{Theorem}
\newtheorem{lemma}{Lemma}
\newtheorem{remark}{Remark}
\newtheorem{corollary}{Corollary}

\begin{document}

\large
\begin{center}
{\bf  On the Bayesian analysis of a non-identifiable Binomial model \footnote{ \today. }}   
\end{center}
\normalsize
\begin{center}
{\sc \'Eric Marchand } \\
{\it   Universit\'e de
    Sherbrooke, D\'epartement de math\'ematiques, Sherbrooke Qc,
    CANADA, J1K 2R1 \quad (e-mail: eric.marchand@usherbrooke.ca) } \\
\end{center}

\vspace*{-0.5cm}
\begin{abstract}
\noindent
We provide analysis for the posterior distribution and expectation of $(p_1, p_2)$ where $Y|p_1,p_2 \sim \hbox{Binomial}(n, p_1 p_2)$ and $ (p_1, p_2)$ is uniformly distributed on the unit square $[0,1]^2$.   We exhibit interesting expressions in terms of a truncated Beta distribution, a finite mixture of Beta distributions and harmonic numbers, and derive a simple large sample size $n$ approximation for the posterior expectations $\mathbb{E}(p_i|y)$ as well as for the normalization constant in the posterior joint density. 
\end{abstract}

\vspace{0.1cm} 


\section{Introduction}

In this note, we consider the Bayesian model:
\begin{equation}
\label{model}
Y|p_1,p_2 \sim \hbox{Binomial}(n, p_1 p_2) \; \hbox{with } \;  p_1, p_2 \sim^{\!\!\!\!\!\!\!\! i.i.d.} \mathcal{U}(0,1)\,, 
\end{equation}
and extract interesting representations relative to the Bayesian posterior distribution and expectation.  The above binomial problem arises in a basic setting with independent, but otherwise arbitrary, events $A$ and $B$ such that $\mathbb{P}(A)\,=\, p_1$ and $\mathbb{P}(B)\,=\, p_2$, and where $n$ independent replicates of the indicator function $\mathbb{I}_{A \cap B}$ are observed.    The inclusion of the uniform prior density appears in a recent paper (Surjanovic et al., 2025), where the authors introduce and illustrate the use of the software package Pigeons.jl,  which is designed to sample from difficult or intractable distributions including Bayesian posterior distributions. The authors refer to the above as a non-identifiable model for a  ``coinflip'' data set since the effects of the parameters $p_1$ and $p_2$ cannot be distinguished.

Although the problem is simply formulated, we report here on interesting features which include: (i) a mixture of Beta distributions representation for the posterior distribution of $p_1|y$, (ii) a weighted Beta distribution; which we introduce and define; for the posterior conditional distribution of $p_2|p_1,y$, (iii) posterior expectations $\mathbb{E}(p_i|y)$ expressed in terms of a harmonic mean, (iv) a large sample size $n$ approximation for the marginal distribution of $Y$ and related normalization constant, and (v) the large sample size $n$ approximation 
\begin{equation}
\mathbb{E}(p_1|Y=\alpha n)\, \approx \, \frac{1-\alpha}{-\ln(\alpha)}\,,
\end{equation}
which depends on $\alpha$ but not $n$.

Although model (1.1) is a simple one and well-formulated for computational evaluations, we believe that concise and elegant distributional findings such as those presented in this work are particularly welcomed, and potentially useful for the validation and implementation study of numerical algorithms. 
\section{Analysis}

\subsection{Marginal distribution and normalization constant}

Clearly, the variables $p_1$ and $p_2$ are interchangeable in model (\ref{model}) and can be equivalently characterized as non-identifiable in the sense that their posterior distributions for a given observation $y$ will necessarily match.   A collapsed version of the model is $Y|w \sim \hbox{Binomial}(n, w)$ with $W$ distributed as the product $p_1 p_2$.   This has density $f(w) \, = \, - \ln (w) \, \mathbb{I}_{(0,1)}(w)$ and moments $\mathbb{E}(W^k) \, = \, \frac{1}{(k+1)^2}, k \geq 0$.  For describing the marginal distribution of $Y$, the model is equivalent to $Y|W$ with $W =^d p_1 p_2$.  One can verify that $W$ has density $- \log (w) \, \mathbb{I}_{(0,1)}(w)$ and moments $\mathbb{E}(W^k) \, = \, \frac{1}{(k+1)^2}, k \geq 0$. With $\mathbb{E}(Y|w)\,=\, nw$ and $\mathbb{V}(Y|w)\,=\, nw(1-w)$, one obtains the marginal expectation and variance $\mathbb{E}(Y)\,=\, \frac{n}{4}$ and $\mathbb{V}(Y) \, = \, \frac{7 n^2}{144} + \frac{5n}{36}$.   The marginal probability mass function of $Y$ is as follows.
\begin{lemma}
\label{lemmamarginalY}
  The marginal probability mass function of $Y$ is given by
\begin{equation}
\label{marginalpmf}
p(y) \, = \, \frac{n!}{y!} \, \sum_{k=0}^{n-y}  \frac{(-1)^k}{k! (n-y-k)!} \, \frac{1}{(y+1+k)^2}\, \mathbb{I}_{\{0,1, \ldots, n\}}(y)\,.
\end{equation}  
{\bf Proof.} We have
\begin{eqnarray*}
p(y) \, & = & \, - \, \int_0^1  \binom{n}{y} \, w^y (1-w)^{n-y} \, \ln(w) \, dw \\
\, & = & \, - \, \binom{n}{y} \, \int_0^1 w^y \sum_{k=0}^{n-y} \, \binom{n-y}{k} \, (-w)^k \, \ln(w) \, dw \\  
\, & = &  \, - \,  \frac{n!}{y!} \, \sum_{k=0}^{n-y} \, \frac{(-1)^k}{k! \, (n-y-k)!} \, \int_0^1  w^{y+k}  \ln(w) \, dw\,,
\end{eqnarray*}
which leads to the result via the evaluation $\, - \int_0^1  w^{y+k} \, \ln(w) \, dw\,=\, \frac{1}{(y+1+k)^2}$. \qed
\end{lemma}
Observe that the posterior density of $(p_1, p_2)$ for a given $y$ is given by 
 $\pi(p_1,p_2|y) \,  =  \frac{1}{p(y)} \, \binom{n}{y} \, (p_1 p_2)^y \, (1- p_1p_2)^{n-y} \, \, \mathbb{I}_{(0,1)^2}(p_1,p_2) $, with the normalization constant $p(y)$ given in Lemma \ref{lemmamarginalY}.   A large sample $n$ approximation for this normalization constant will be obtained in section \ref{sectionalternative}.
 
\subsection{Posterior distribution of $(p_1, p_2)$}

To pursue with the posterior distributions, let us define a `weighted' Beta distribution as follows.
\begin{definition}
\label{definition}
A weighted Beta distribution with parameters $a>0, b>0, \lambda \in [0,1]$, denoted 
$Z \sim \mathbb{B}(a,b,\lambda)$ has density
\begin{equation}
\label{densityweightedbeta}
f(z) \, = \,   \frac{1}{C(a,b,\lambda)} \, z^{a-1} (1-\lambda z)^{b-1} \, \mathbb{I}_{(0,1)}(z)\,.
\end{equation}
\end{definition}
The above family of distributions includes standard Beta distributions $\mathbb{B}(a,b,1)$, reduce to $\mathbb{B}(a,1,1)$ distributions for $b=1$ \footnote{But these distributions are not identifiable in terns of $\lambda$.}, and contains $\mathbb{B}(a,1,1)$ distributions when $\lambda = 0$.  Moreover, these distributions are related through truncation and multiplication as
\begin{equation}
\label{representationconditional+truncation}
\frac{Z}{\lambda}| Z \leq \lambda =^d \mathbb{B}(a,b,\lambda) \,,
\end{equation}
for $Z \sim \mathbb{B}(a,b,1)$ and $\lambda>0$.
   
Notice also that the densities in Definition \ref{definition} have increasing monotone likelihood ratio (mlr) in $Z$ with parameter $\lambda$ for $b <1$, and decreasing mlr in $Z$ with parameter $\lambda$ for $b>1$.      Here are a couple of useful expressions for Definition \ref{densityweightedbeta}'s  normalization constant $C(a,b,\lambda)$ which we will require.
 
\begin{lemma}  We have for $a>0, b \in  \mathbb{N}_+$, and $\lambda \in [0,1]$:
\label{lemmaC}
\begin{equation}
\label{C1}
C(a,b,\lambda) \, = \, \Gamma(a) \, \Gamma(b) \, \sum_{k=0}^{b-1}  \, \frac{\lambda^{b-1-k} \, (1-\lambda)^k }{k! \, \Gamma(a+b-k)} \,,
\end{equation}
and
\begin{equation}
\label{C2}
C(a,b,\lambda) \, = \, \Gamma(a) \, \Gamma(b) \, \sum_{k=0}^{b-1}  \, \frac{\lambda^{k} \, (1-\lambda)^{b-1-k} }{\Gamma(b-k) \,\, \Gamma(a+k+1)}\,.
\end{equation}
\end{lemma}
{\bf Proof.}   Equation (\ref{C2}) follows from (\ref{C1}) with the change of variables $k \to b-1-k$ in the sum.   For (\ref{C1}), set $\rho=\frac{1-\lambda}{\lambda}$.   A binomial expansion, an interchanging of integral and sum, and a Beta integral yield:
\begin{eqnarray*}
C(a,b,\lambda) \, & = & \,  \int_0^1  z^{a-1} (1-\lambda z)^{b-1} \, dz \\
\, & = & \, \int_0^1  z^{a-1} \, \lambda ^{b-1} \, (1-z+\rho)^{b-1} \, dz \\
\, & = & \, \int_0^1  z^{a-1} \, \lambda ^{b-1} \, \sum_{k=0}^{b-1} \binom{b-1}{k}\, \rho^k \, (1-z)^{b-1-k} \,  dz \\
\, & = & \,  \lambda^{b-1} \, \sum_{k=0}^{b-1} \rho^k \, \binom{b-1}{k} \, \frac{\Gamma(a) \, \Gamma(b-k)}{\Gamma(a+b-k)}\,,
\end{eqnarray*}
which is equal to (\ref{C1}).  \qed

\begin{theorem}
\label{theorem}
For $Y|p_1,p_2 \sim \hbox{Binomial}(n, p_1 p_2)$ and uniform prior for $(p_1,p_2)$ on $[0,1]^2$, the posterior distribution of $(p_1,p_2)$ is interchangeable and admits the decomposition: 
\begin{equation}
\label{equationbetamixture}
(i) \; p_1|p_2,y \sim \mathbb{B}(y+1, n-y+1, p_2) \, \hbox{ with density } \,
(ii) \, \pi(p_2|y) \, = \, \sum_{k=0}^{n-y}  w_k \, f_{k+y+1, n-y-k+1}(p_2)\,,
\end{equation}
with $f_{a,b}$ being a $\mathbb{B}(a,b,1)$ density, and probabilities $w_k$ such that  
$w_k \propto (k+y+1)^{-1} \, \mathbb{I}_{\{0, \ldots, n-y\}}(k)$. 
\end{theorem}
{\bf Proof.}   We have for the joint posterior density
\begin{eqnarray*}
\pi(p_1,p_2)|y \, & \propto & \,  \binom{n}{y} \, (p_1 p_2)^y \, (1- p_1p_2)^{n-y} \, \, \mathbb{I}_{(0,1)^2}(p_1,p_2) \\
 \, & \propto & \, \frac{p_1^y \, (1- p_1 p_2)^{n-y}}{C(y+1, n-y+1, p_2)} \,\, p_2^y  \, C(y+1, n-y+1, p_2) \,,
\end{eqnarray*}
which implies (i).   Moreover, we have from (\ref{C2})
\begin{eqnarray*}
\pi(p_2 |y)  \, & \propto \, &  p_2^y  \, \, C(y+1, n-y+1, p_2) \\
\, & \propto \, & \sum_{k=0}^{n-y}  \frac{\Gamma(n+2) \, p_2^{y+k} \, (1-p_2)^{n-y-k}}{\Gamma(n-y-k+1) \, \Gamma(y+k+1)} \,  \	\times \frac{1}{y+k+1} \,,
\end{eqnarray*}
which yields (\ref{equationbetamixture})\,. \qed

Theorem \ref{theorem} holds exactly for all $(n,y)$ and establishes a Beta mixture representation for the marginal posterior distributions of $p_2$ and $p_1$.   It is a discrete mixture with mixing variable taking on $n-y+1$ values with probabilities proportional to consecutive elements of a harmonic sequence.  

\begin{remark} 
\label{remarksimulation}
Theorem \ref{theorem} suggests two simple options for simulating from the posterior distribution of $(p_1,p_2)$.  On one hand, the full conditionals are available for the Gibbs sampler, i.e., (i) and its twin version $p_2|p_1,y \sim \mathbb{B}(y+1, n-y+1, p_1) \,$ which requires simulating from a weighted Beta distribution, which in turn can be done through its cdf, or via representation (\ref{representationconditional+truncation}) and a truncated Beta distribution.   Alternatively, parts (i) and (ii) combined generates exactly one observation under the posterior, with generators of a weighted Beta distributed variable, the discrete variable $K$ according to the probabilities $w_k$, and a Beta distributed $B(k+y+1, n-y-k+1)$ variable.    
\end{remark}
 We point out that the case $y=n$, reminiscent of Laplace's celebrated problem of succession where the product $U=p_1p_2$ is uniformly distributed on $(0,1)$, simplifies to independent components $p_1$ and $p_2$ under the posterior distribution, both distributed as $\mathbb{B}(n+1,1)$. Coincidentally, in Laplace's problem, one also obtains a $\mathbb{B}(n+1,1)$ distribution for the posterior distribution of $U|n$.    However, the above coincidence is a manifestation of a more general property, specific to the posterior distribution for $Y=n$, where $Y|p_1, \dots, p_k \sim \hbox{Binomial}\big(n, \prod_{i=1}^k p_i\big)$ and $p_1, \dots, p_k$ has prior density of the form $\prod_{i=1}^k g_i(p_i)$.  Indeed, for the above situation, one obtains directly independently distributions components $p_i$ under the posterior, with densities
\begin{equation}
\pi(p_i|n) \, \propto   \, p_i^n \, g_i(p_i) \, \hbox{ for } i=1,\dots, k\,,
\end{equation}
also irrespectively of $k$.

\subsection{Posterior expectation}

We now turn our attention towards the posterior expectation $\mathbb{E}(p_2|y)$ with  the following representation bringing into play a harmonic mean and leading to a simple large $n$ approximation.
\begin{corollary}
\label{posteriorexpectation}
Consider model (\ref{model}) and suppose that $Y=y$ has been observed.
\begin{enumerate}
\item[ (a)]   The marginal posterior expectation of $p_2$ (or $p_1$) is given by
\begin{equation}
\label{expectationharmonic1}   \mathbb{E}(p_2|y) \, = \, \frac{1}{n+2} \; \mu_H\{y+1, y+2, \ldots, n+1  \}\,,
\end{equation}
or, alternatively,
\begin{equation}
\label{expectationharmonic2}
\mathbb{E}(p_2|y) \, = \, \frac{(n-y+1)/(n+2)}{H_{n+1} - H_y}, 
\end{equation}
where $\mu_H$ is the harmonic mean given by $\mu_H\{a_1, \ldots, a_N  \}\, = \, \big\{\frac{1}{N} \sum_{i=1}^N a_i^{-1}   \big\}^{-1}$ for positive real numbers  $a_1, \ldots, a_N$, and where $H_m\,=\,1+ \frac{1}{2}+ \cdots + \frac{1}{m}$ is the $m^{th}$ harmonic number. 

\item[ (b)]  For $\alpha \in (0,1)$, we have the liming expression
\begin{equation}
\label{expectationlimiting}
\lim_{n \to \infty} \mathbb{E}\big(p_2|Y=\,\lfloor \alpha n \rfloor \big) \, = \frac{1-\alpha}{- \ln (\alpha)}\,,
\end{equation}
where $\lfloor a \rfloor$ is the integer part of $a>0$.
\end{enumerate}
{\bf Proof.}
\begin{enumerate}
\item[ (a)]   Expression (\ref{expectationharmonic1}) follows directly from (\ref{equationbetamixture}) with  $\mathbb{E}(p_2|y) \, = \,  \frac{\sum_{k=0}^{n-y}  w_k \, \frac{k+y+1}{n+2}}{\sum_{k=0}^{n-y}  w_k}$ and   
$w_k \, \propto \, (k+y+1)^{-1}$, while (\ref{expectationharmonic2}) follows from part (a) since $\mu_H\{y+1, y+2, \ldots, n+1  \}\,=\,  \frac{n-y+1}{H_{n+1}-H_y}$.

\item[ (b)]  Since $\lim_{m \to \infty} \big\{H_m - \ln (m)\big\} = \gamma$, $\gamma = 0.577216..$ being Euler's constant, we obtain from (\ref{expectationharmonic2})
\begin{align}
\lim_{n \to \infty} \mathbb{E}(p_2|Y\, = \lfloor \alpha n\rfloor  \,) \, & = \, \frac{\lim_{n \to \infty} (n-\lfloor \alpha n\rfloor +1)/(n+2)}{\lim_{n \to \infty} \big(H_{n+1} - H_{\lfloor \alpha n\rfloor }\big)} \, \\
\,& =  \lim_{n \to \infty} \big\{\frac{1-\alpha}{\ln(n+1) - \ln(\lfloor \alpha n\rfloor) }\big\} =\frac{1-\alpha}{- \ln (\alpha)}\,.  \qed 
\end{align}
\end{enumerate}
\end{corollary}
The above makes large sample approximations straightforward.  For $n= 100 000$ and $y= 50 000$ as evoked by Surjanovic et al. (2025), or more generally with $Y= \frac{n}{2}$, we have $\mathbb{E}(p_2|Y=\frac{n}{2}) \approx  \frac{1}{2 \ln 2} \,=\,0.721348...$ for large $n$.

\subsection{Alternative approach and a large $n$ approximation for the normalization constant}
\label{sectionalternative}

An alternative approach is to reparametrize as $\theta=p_1$ and $\theta_2=p_1 p_2$.  In this case, all of the prior parametric information for the model resides with $\theta_2$, and one obtains the densities
\begin{equation}
\nonumber  \pi(\theta_1|\theta_2, y) = -\frac{1}{\theta_1 \ln (\theta_2) }\, \mathbb{I}_{(\theta_2,1)}(\theta_1),
\end{equation}
independently of $y$, and 
\begin{equation}
\nonumber
\pi(\theta_2|y) \,= \, \frac{\binom{n}{y} \,\theta_2^y \, (1-\theta_2)^{n-y} \,\, (-\ln \theta_2)}{p(y)} \, \, \mathbb{I}_{(0,1)}(\theta_2)\,.
\end{equation}
In terms of the posterior expectation, observe that $\mathbb{E}(\theta_1|\theta_2, y) \, = \, \frac{\theta_2 -1}{\ln \theta_2}$ which leads to
\begin{eqnarray}
\nonumber \mathbb{E}(p_1|y) \, = \, \mathbb{E}(\theta_1|y) \, & = &\, \mathbb{E} \big(\frac{\theta_2 -1}{\ln \theta_2}  \, \big| y \big) \\
\nonumber & = &\, \binom{n}{y} \, \,\frac{ \int_0^1   \theta_2^y \, (1-\theta_2)^{n-y+1}  \, d\theta_2}{p(y)} \\
\label{expressionp(y)} & = &\, \frac{n-y+1}{(n+1) \, (n+2) \, p(y)}\,.
\end{eqnarray}
To conclude, we return to the findings of Surjanovic et al. (2025) and the various numerical evaluations and simulations which they provide for the problem considered here.   Namely, with a central issue in Bayesian computing being the evaluation of the normalization constant, or its logarithm, for a given $(n,y)$ given by $p(y)$ in (\ref{marginalpmf}).   Making use of (\ref{expressionp(y)}),  (\ref{expectationharmonic1}), and (\ref{expectationlimiting}), we obtain the exact value
\begin{eqnarray*}
p(y) \,  = \,  \frac{n-y+1}{(n+1) \, \mu_H\{y+1, \dots, n+1 \}}\,,
\end{eqnarray*}
and for $y=\alpha n$ the large samples size $n$ approximation
\begin{equation}
\nonumber
p(y) \, \approx   \frac{n-n\alpha+1}{(n+1) (n+2) \, \mathbb{E}(p_1|y) } \approx \frac{- \ln(\alpha)}{n}\,.
\end{equation}
As an illustration for $y=50 000$ and $n=100 000$, the above yields   $\ln p(50 000) \, \approx \, \ln\big\{(\ln 2) \times 10^{-5} \big\} \, = \, -11.879438..$ which corroborates Surjanovic et al. (2025)'s given value $\ln p(50 000) \approx -11.8794...$.
Finally, we point out that the marginal probability mass function of $Y$ is well-approximated by the above probability density function (pdf) $p(y)$ viewed as a function of $y$ with $\alpha=y/n$, which also implies that $X=Y/n$ has approximating pdf  $f(x)\,=\, -\ln(x) \, \mathbb{I}_{(0,1)}(x)$.  This is indeed the prior density of $p_1 p_2$, and the matching follows also since $X|p_1, p_2$ converges in probability to $p_1 p_2$, so that $X$ marginally converges in distribution to that of $p_1 p_2$.

\section*{Acknowledgements}
The author thanks Alexandre Bouchard-C\^ot\'e and Jim Berger for useful exchanges. The author is grateful for research support from NSERC of Canada.

\end{document}